\documentclass{birkjour}
 \newtheorem{theorem}{Theorem}[section]
 
 \newtheorem{lemma}[theorem]{Lemma}
 \newtheorem{proposition}[theorem]{Proposition}
 \theoremstyle{definition}
 \newtheorem{definition}[theorem]{Definition}
 \theoremstyle{remark}
 \newtheorem{remark}[theorem]{Remark}
 
 \numberwithin{equation}{section}

  \usepackage{cite}
\usepackage{hyperref}
\hypersetup{
	colorlinks,
	linkcolor = {blue},
	citecolor = {blue},
	filecolor = {blue},
	urlcolor = {blue} 
}

\begin{document}

\title[A New Lemoine-Type Construction]
 {A New Lemoine-Type Circle Construction}

%----------Author 1
\author[Miłosz Płatek]{Miłosz Płatek\footnote{Independent researcher,
Kraków, Poland, milosz@platek.org}}

\address{}

\email{}


%%% ----------------------------------------------------------------------

\begin{abstract}
This paper presents a new Lemoine-type construction arising from a configuration of six concyclic points. We prove the concyclicity result, establish a converse theorem, and relate the resulting circle to previously known Lemoine circles, in particular the circle introduced by Q. T. Bui. We also note that the resulting circle does not belong to the family of Tucker circles.
\end{abstract}

%%% ----------------------------------------------------------------------
\maketitle
%%% ----------------------------------------------------------------------
%\tableofcontents
\section{Historical Background}

In 1873, Émile Lemoine unveiled the symmedian point in his influential work On a Remarkable Point of the Triangle \cite{lem}. Lemoine’s concise yet profound study is often recognized as the catalyst for what became known as modern triangle geometry, introducing key concepts that guided generations of synthetic investigations.

Over the next decades, mathematicians pursued an array of discoveries in triangle centers, loci, and constructs, culminating in the extensive compilation found in Klein’s Encyclopedia of Mathematical Sciences (1914) \cite{klein}. However, as the 20th century progressed, the focus shifted toward algebraic and analytic frameworks, causing synthetic triangle geometry to recede from prominence.

Toward the end of the 20th century, renewed interest in triangle geometry emerged, facilitated by advances in dynamic geometry software and the increasing availability of online discussion platforms focused on synthetic methods. One such forum, Hyacinthos, an online discussion group devoted to triangle geometry, fostered an environment for exploration and new discoveries. The group’s name honors Émile-Michel-Hyacinthe Lemoine, whose contributions laid the foundation for the field’s development. Among its many outcomes were the Third Lemoine Circle (Jean-Pierre Ehrmann, 2002)\cite{pierre, ehrmann} and the Q.T. Bui Circle (2006) \cite{bui}, both derived from and inspired by Lemoine’s original ideas. In parallel, Clark Kimberling’s Encyclopedia of Triangle Centers \cite{kimberling} systematized tens of thousands of triangle centers and related objects.

This paper continues Lemoine's legacy. We present a new Lemoine-type construction, provide synthetic proofs of its properties, establish a converse theorem, and compare the resulting circle with known Lemoine circles by examining their arrangement along the Brocard axis \cite{brocardaxis}.

We note that the circle considered here coincides with the Dao-symmedial circle, whose center is $X(5092)$ \cite{kimberling}. Consequently, for a triangle in general position, combining the two constructions yields a new configuration of twelve distinct concyclic points and reveals a connection between two otherwise different geometric constructions. This identification can be verified by a direct but lengthy computation, which we omit.

\section{The Family of Tucker Circles}

A precise understanding of the definition of the Family of Tucker Circles\cite{tucker}\cite{wolf} is essential for the analysis of known Lemoine circles. Throughout this paper, $L$ denotes the Lemoine point of triangle $ABC$, defined as the common point of its three symmedians, and $O$ denotes the circumcenter of triangle $ABC$. Moreover, we say that a line $k$ is antiparallel to a line $l$ with respect to an angle $\angle XYZ$ if the reflection of $k$ across the angle bisector of $\angle XYZ$ is parallel to $l$. In Definition~\ref{tuck}, antiparallels are taken with respect to the angle of triangle $ABC$ opposite the side under consideration. The definition of Tucker circles is as follows.

\begin{definition}[Tucker circle]\label{tuck}
Let $ \triangle ABC $ be a triangle. Let the point $B_a$ lie on the line $AB$, distinct from the points $B$ and $A$. Let:
\begin{itemize}
    \item the line antiparallel to $BC$ through point $B_a$ intersect $AC$ at point $C_a$;
    \item the line parallel to $AB$ through point $C_a$ intersect $BC$ at point $C_b$;
    \item the line antiparallel to $AC$ through point $C_b$ intersect $BA$ at point $A_b$;
    \item the line parallel to $BC$ through point $A_b$ intersect $CA$ at point $A_c$;
    \item the line antiparallel to $AB$ through point $A_c$ intersect $BC$ at point $B_c$.
\end{itemize}
Then the line parallel to $AC$ through point $B_c$ intersects $AB$ at point $B_a$. Thus, the sides of the hexagon $B_aC_aA_cB_cC_bA_b$ are alternately parallel and antiparallel to the sides of triangle $ABC$.  
Such a hexagon is called a Tucker hexagon.
\end{definition}

\begin{proposition}\label{prop}
The vertices of the Tucker hexagon lie on a single circle, called the Tucker circle, whose center lies on the line $OL$.
\end{proposition}

\noindent All currently known Lemoine circles belong to the family of Tucker circles. This significantly simplifies their proofs---it suffices to show that consecutive sides are alternately parallel and antiparallel to the corresponding sides of the triangle, by virtue of Claim \ref{prop}. 

\section{Known Lemoine Circles}

In this section, we focus on circles associated with the point \emph{$L$}. Émile Lemoine conducted extensive research on this point in 1873, defining what are now known as the \emph{First} and \emph{Second Lemoine Circles} \cite{history, ehrmann}.

\begin{proposition}
(First Lemoine Circle) Let the lines parallel to the lines $BC$, $BC$, $CA$, $CA$, $AB$, $AB$ through the point $L$ intersect the lines $CA$, $AB$; $AB$, $BC$, $BC$, $CA$ at six points. These six points lie on a single circle, called the \emph{First Lemoine Circle} of triangle $ABC$, and its center is the midpoint of the segment $LO$.
\end{proposition}

\begin{figure}[h]
    \centering
    \includegraphics{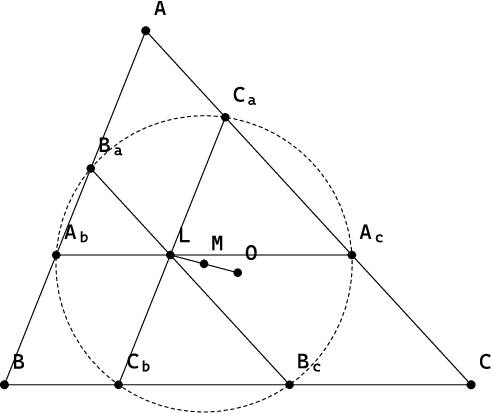}
    \caption{First Lemoine Circle}
    \label{fig:placeholder}
\end{figure}

\begin{proposition}
(Second Lemoine Circle) Let the lines through point $L$, which are antiparallel to the lines $BC$, $CA$, and $AB$, intersect lines $CA$, $AB$, and $BC$ in six points. These six points lie on a single circle, called the \emph{Second Lemoine Circle} of triangle $ABC$, and its center is point $L$.
\end{proposition}

\begin{figure}[h]
    \centering
    \includegraphics{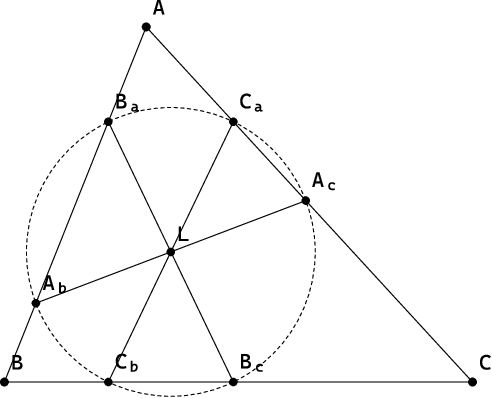}
    \caption{Second Lemoine Circle}
    \label{fig:placeholder}
\end{figure}

\begin{figure}[h]
    \centering
    \includegraphics{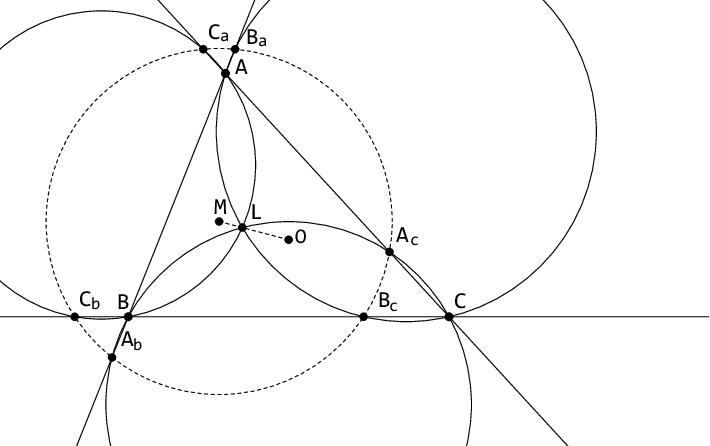}
    \caption{Third Lemoine Circle}
    \label{fig:placeholder}
\end{figure}

These two circles gave rise to the family of Lemoine circles. Their distinguishing feature is the fact that these six points lie on a single circle if and only if the initial point is the Lemoine point. Moreover, the center of each of these circles lies on the line $OL$, the Brocard axis of the triangle $ABC$. Consequently, mathematicians became interested in searching for circles with similar properties, obtaining two significant results.

\begin{proposition}
   (Third Lemoine Circle) Let the circumcircle of triangle $BLC$ intersect lines $AB$ and $AC$ at points $A_b$ and $A_c$, respectively (distinct from $B$ and $C$). Let the circumcircle of triangle $CLA$ intersect lines $BC$ and $BA$ at points $B_c$ and $B_a$, respectively (distinct from $C$ and $A$). Let the circumcircle of triangle $ALB$ intersect lines $CA$ and $CB$ at points $C_a$ and $C_b$, respectively (distinct from $A$ and $B$). Then, the points $A_b$, $A_c$, $B_a$, $B_c$, $C_a$, and $C_b$ lie on a single circle, and its center $M$ lies on line $OL$ and satisfies the relation $LM = -\frac{1}{2} LO$ (where the segments are directed).
\end{proposition}

\begin{proposition}
    (Q.T.Bui Circle) Let $ABC$ be a triangle with circumcircle $\omega$. The circle $\omega_1$, passing through points $A$ and $L$ and tangent to $\omega$ at $A$, intersects lines $AB$ and $AC$ at points $A_b$ and $A_c$, respectively. The circle $\omega_2$, passing through points $B$ and $L$ and tangent to $\omega$ at $B$, intersects lines $BC$ and $BA$ at points $B_c$ and $B_a$, respectively. The circle $\omega_3$, passing through points $C$ and $L$ and tangent to $\omega$ at $C$, intersects sides $CA$ and $CB$ at points $C_a$ and $C_b$, respectively. Then, the points $A_b$, $A_c$, $B_a$, $B_c$, $C_a$, and $C_b$ lie on a single circle, and its center $M$ lies on line $OL$ and satisfies the relation $LM = \frac{1}{4} \cdot LO$.
\end{proposition}

\begin{figure}[h]
    \centering
    \includegraphics{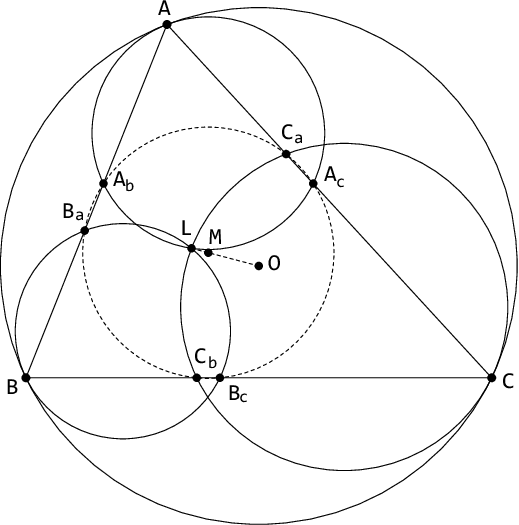}
    \caption{Q.T.Bui Circle}
    \label{fig:placeholder}
\end{figure}

\begin{remark}
    Once again, for these circles, the only initial point for which these points lie on a single circle is point $L$, and in that case, the centers of those circles also lie on the Brocard axis.
\end{remark}

\section{A New Lemoine-Type Construction}

Here, we present the main result: a new Lemoine-type circle construction that is closely related to the construction introduced by Q. T. Bui in 2006. Unlike the Lemoine circles considered above, the resulting circle does not belong to the family of Tucker circles. This distinction makes the proof considerably more complex.

\begin{theorem}\label{new}
    Let $ABC$ be a triangle with a circumcircle $\omega$. Let $A', B', C'$ denote the intersections of the lines $AL, BL, CL$ with the circumcircle $\omega$ (distinct from $A, B, C$). The circle $\omega_1$, passing through points $A'$ and $L$, tangent to $\omega$ at $A'$, intersects the line $BC$ at points $A_b$ and $A_c$. The circle $\omega_2$, passing through points $B'$ and $L$, tangent to $\omega$ at $B'$, intersects the line $CA$ at points $B_c$ and $B_a$. The circle $\omega_3$, passing through points $C'$ and $L$, tangent to $\omega$ at $C'$, intersects the line $AB$ at points $C_a$ and $C_b$. Then, the points $A_b$, $A_c$, $B_a$, $B_c$, $C_a$, $C_b$ lie on a single circle. This circle is not a Tucker circle, and its center $M$ lies on the line $OL$ and satisfies $LM = \frac{3}{4}\cdot LO$.
\end{theorem}

    \begin{figure}[h]
        \centering
        \includegraphics{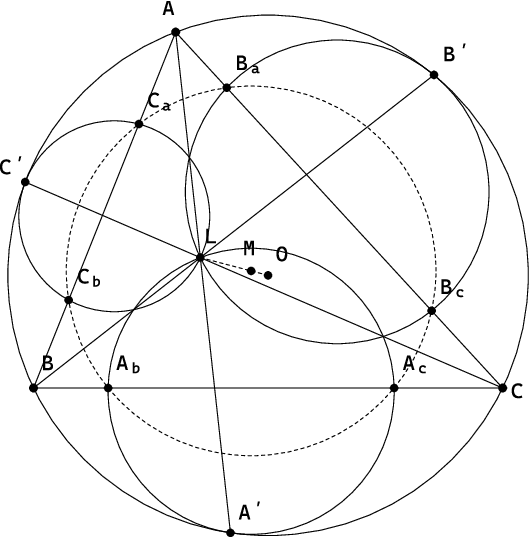}
        \caption{Theorem~\ref{new}.}
        \label{fig:placeholder}
    \end{figure}
    
\begin{proof}
    We will use the following lemma, which appears as Exercise 7.7 in Honsberger's book \cite[p.~77]{history}:
        \begin{lemma}\label{lem42}
            Point $L$ is the Lemoine point in triangle $A'B'C'$. 
        \end{lemma}
        
    Let us denote by $X$ the point of intersection of the tangents at points $B'$ and $C'$ to the circle $\omega$. Then $XB' = XC'$, and consequently, point $X$ lies on the radical axis of circles $\omega_2$ and $\omega_3$. Since point $L$ also lies on the radical axis of these two circles, the line $LX$ is the radical axis of circles $\omega_2$ and $\omega_3$, and thus point $A$ also lies on this axis (because $A', L, A, X$ are collinear by the definition of the symmedian). Therefore, from the power of a point, we get that
    \begin{equation*}
        AB_a \cdot AB_c = AC_a \cdot AC_b,
    \end{equation*} 
    so the quadrilateral $B_aB_cC_bC_a$ is cyclic. Similarly, we obtain that the quadrilaterals $B_aB_cA_cA_b$ and $A_cA_bC_bC_a$ are cyclic. If any two of the three proven cyclic quadrilaterals are concentric, we immediately conclude that the points $A_b$, $A_c$, $B_c$, $B_a$, $C_a$, $C_b$ lie on one circle. Let us assume, therefore, that they have pairwise different centers. Then the radical axes of these circles are the corresponding sides of triangle $ABC$, which obviously do not intersect at a single point and are not pairwise parallel, a contradiction with the theorem on radical axes. Therefore, the points $A_b$, $A_c$, $B_a$, $B_c$, $C_a$, $C_b$ lie on one circle.

    \begin{figure}[h]
        \centering
        \includegraphics{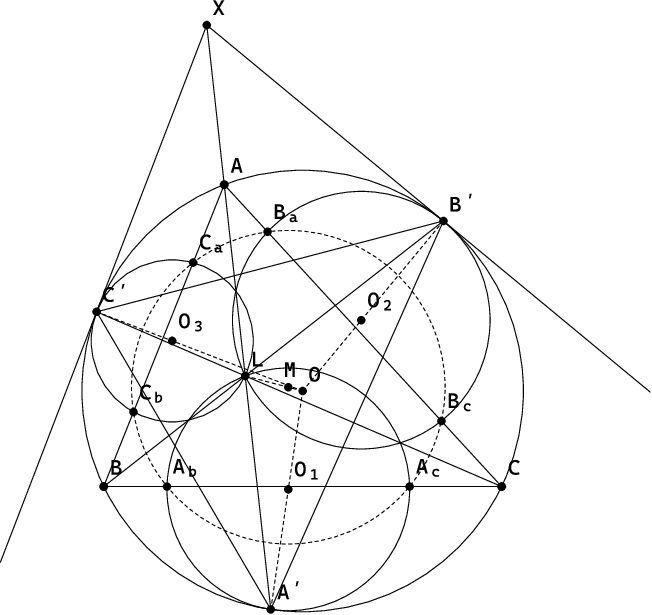}
        \caption{Points $A_b$, $A_c$, $B_a$, $B_c$, $C_a$, $C_b$ are concyclic.}
        \label{fig:placeholder}
    \end{figure}
    
We will now show that point $M$ lies on $LO$ and satisfies 
\begin{equation*}
    LM = \frac{3}{4}\cdot LO.
\end{equation*}

We will prove that the perpendicular bisector of $A_bA_c$ intersects the line $LO$ at a point $M'$ with the ratio stated above. Let $\omega_a$ be the circle with center $O_1'$ passing through $A$ and $L$ and tangent to $\omega$. Let $A_b'$ and $A_c'$ be the intersection points of $\omega_a$ with the lines $AB$ and $AC$, respectively (see Fig\ref{fig9}).

\begin{figure}[h]
        \centering
        \includegraphics{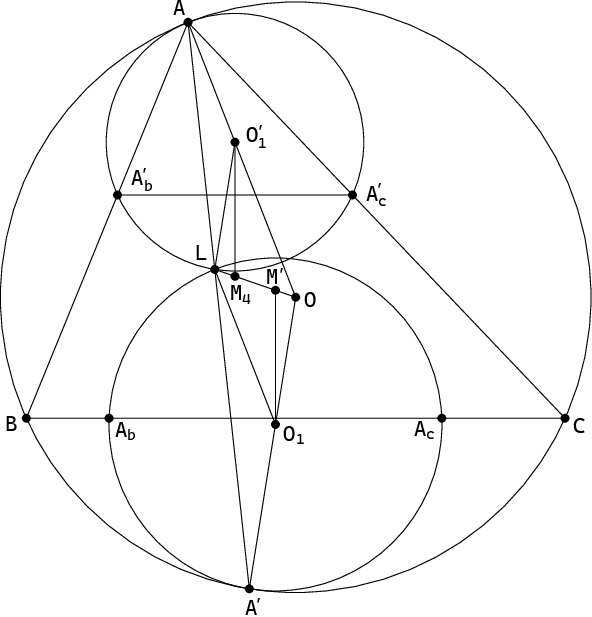}
        \caption{$LM_4= M'O$}
        \label{fig9}
    \end{figure}
\end{proof}

The circle $\omega_a$ comes from the construction of the Q.T. Bui circle. Hence, $A_b'A_c' \parallel BC$. Let $M_4$ be the center of the Q.T. Bui circle. Then
\begin{equation*}
LM_4=\frac{1}{4}LO.
\end{equation*}
Moreover, by the definition of the point $M'$, we have $M'O_1 \perp BC$. The line $O_1'M_4$ is the perpendicular bisector of $A_b'A_c'$. Therefore,
\begin{equation*}
(M_4O_1' \perp A_b'A_c' \parallel BC \perp M'O_1) \implies M_4O_1' \perp M'O_1.
\end{equation*}
Consider the homothety that maps the circle $\omega'$ to the circle $\omega$. It follows that $LO_1' \parallel OO_1$. Similarly, we can prove that $LO_1 \parallel O_1'O$. Hence, the quadrilateral $LO_1OO_1'$ is a parallelogram. Since $O_1'M_4 \parallel O_1M'$, by symmetry with respect to the midpoint of $LO$ we obtain $LM_4=M'O$. Thus,
\begin{equation*}
LM' = \frac{3}{4}\cdot LO.
\end{equation*}
Similarly, the perpendicular bisectors of $B_aB_c$ and $C_aC_b$ meet the line $LO$ at the point $M'$ with the same ratio. Therefore, $M=M'$, which proves that $M$ lies on the line $OL$ and satisfies the desired ratio.

\section{Converse Theorem} Here, we complete our analysis by presenting the converse theorem of our main result.

\begin{figure}[h]
    \centering
  \includegraphics{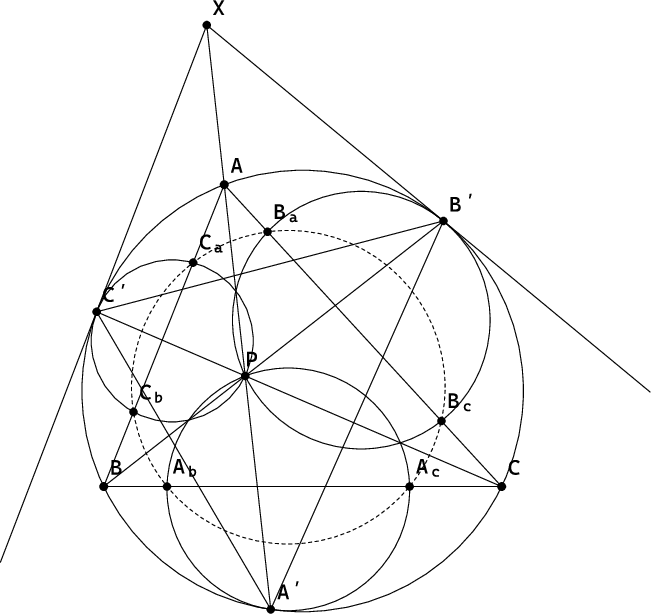}
    \caption{Proof of Theorem~\ref{conv}.}        \label{fig:placeholder}
\end{figure}

\begin{theorem}\label{conv}
    Let $P$ be any point inside triangle $ABC$ with a circumcircle $\omega$. Let $A',B',C'$ denote the intersections of the lines $AP, BP, CP$ with the circle $\omega$ (distinct from $A,B,C$). The circle $\omega_1$ passing through points $A'$ and $P$, tangent to $\omega$ at point $A'$, intersects side $BC$ at points $A_b$ and $A_c$. The circle $\omega_2$ passing through points $B'$ and $P$, tangent to $\omega$ at point $B'$, intersects side $CA$ at points $B_c$ and $B_a$. The circle $\omega_3$ passing through points $C'$ and $P$, tangent to $\omega$ at point $C'$, intersects side $AB$ at points $C_a$ and $C_b$. Then, if the points $A_b$, $A_c$, $B_c$, $B_a$, $C_a$, and $C_b$ lie on one circle, point $P$ is the Lemoine point of triangle $ABC$.
\end{theorem}

\begin{proof}
    We will use the same notation as in the proof of Theorem \ref{new}. Let $P$ be such a point that the points $A_b$, $A_c$, $B_c$, $B_a$, $C_a$, and $C_b$ lie on one circle. Then, by the power of a point, we get that: $AB_a \cdot AB_c = AC_a \cdot AC_b$, so $A$ lies on the radical axis of the circles $\omega_2$ and $\omega_3$. Thus, $AP$ is the radical axis of these circles, and in particular, points $A'$ and $X$ lie on their radical axis. Therefore, the points $X$, $A$, $P$, and $A'$ are collinear (they lie on the radical axis of $\omega_2$ and $\omega_3$). Hence, point $P$ lies on the symmedian drawn from vertex $A'$ in triangle $A'B'C'$ (directly from the construction of the symmedian in a triangle). Similarly, we conclude that point $P$ lies on the symmedian drawn from vertex $B'$ and the symmedian drawn from vertex $C'$ in triangle $A'B'C'$. Therefore, point $P$ is the Lemoine point of triangle $A'B'C'$. From the Lemma~\ref{lem42}, applied to triangle $A$, $B$, $C$, we conclude that point $P$ is also the Lemoine point of triangle $ABC$.
\end{proof}

\section{Lemoine Circle Centers Along the Brocard Axis}

\begin{figure}[h]
        \centering
        \includegraphics{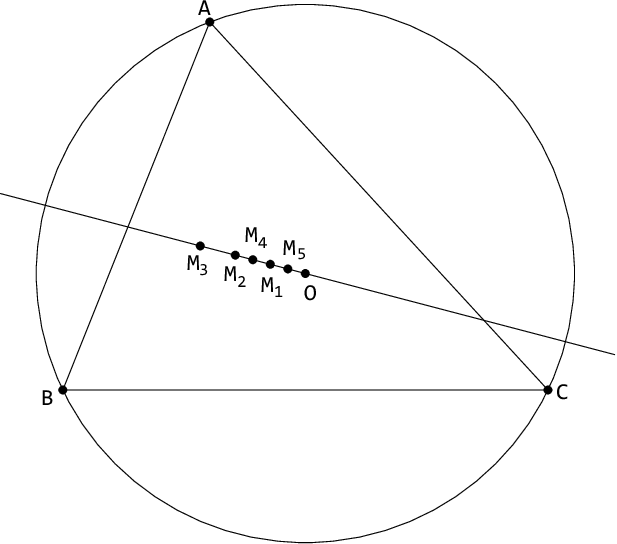}
        \caption{Lemoine Circle Centers Along the Brocard Axis}
        \label{fig:placeholder}
\end{figure}
It is worth noting how the centers of all these circles divide the line $OL$. Here $M_1$ is the center of the \emph{First Lemoine Circle}, $M_2$ is the center of the \emph{Second Lemoine Circle}, $M_3$ is the center of the Third Lemoine Circle, $M_4$ is the center of Q.T. Bui's Circle, and $M_5$ is the center of our Lemoine Circle. To achieve full symmetry, only a Lemoine Circle centered at the midpoint of segment $M_2M_3$ is missing. No such Lemoine Circle has been discovered yet, leaving room for further research.

% ------------------------------------------------------------------------
\end{document}